\documentclass{amsart}
\usepackage{amscd}
\usepackage{amsmath, amssymb}
\usepackage{amsfonts}
\usepackage{enumerate}

\DeclareMathOperator*{\vol}{\ensuremath{vol}}

\numberwithin{equation}{section}

\theoremstyle{plain}
\newtheorem{theorem}{Theorem}
\newtheorem{defn}{Definition}
\newtheorem{lemma}{Lemma}
\newtheorem{remark}{Remark}

\newtheorem{cor}{Corollary}
\newtheorem{pro}{Proposition}
\newtheorem{example}{Example}

\begin{document}

\title[Rigidity of Hawking mass and isoperimetric surfaces]{Rigidity of Hawking mass for surfaces in three manifolds}
\author{Jiacheng Sun}
\address{Beijing International Center of Mathematics Research, Peking University, Yiheyuan Road 5, Beijing, P.R.China, 100871}
\email{sunxujason@pku.edu.cn}
\maketitle

\begin{abstract}
It is well-know that Hawking mass is nonnegative for a stable constant mean curvature ($CMC$) sphere in three manifold of nonnegative scalar curvature. R. Bartnik proposed the rigidity problem of Hawking mass of stable $CMC$ spheres. In this paper, we show partial rigidity results of Hawking mass for stable $CMC$ spheres in asymptotic flat ($AF$) manifolds with nonnegative scalar curvature. If the Hawking mass of a nearly round stable $CMC$ surface vanishes, then the surface must be standard sphere in $\mathbb{R}^3 $ and the interior of the surface is flat. The similar results also hold for asymptotic hyperbolic manifolds. A complete AF manifold has small or large isoperimetric surface with zero Hawking mass must be flat. We will use the mean-field equation and monotonicity of Hawking mass as well as rigidity results of Y. Shi in our proof.
\end{abstract}

\section{Introduction}

One of the most important tasks in general relativity is to understand the mass of space-time. The first attempt on this topic is Positive Mass Theorem, which says that the mass of asymptotic flat manifold is nonnegative if the scalar curvature is nonnegative, and the mass vanishes if and only if the manifold is isometric to standard Euclidean space. Another important attempt is Penrose Inequality which tells us that the mass is no less than $\sqrt{\frac{A}{16\pi}}$ when there is a horizon, where $A$ is the area of outmost minimal surface, the equality holds if and only if the manifold is isometric to Schwarzschild space. From Penrose Inequality we see the impact of boundary behavior is also remarkable. This motivates us to study quasi-local mass for a compact manifold with boundary.

Brown-York mass is a well defined quasi-local mass for domain with convex boundary, which characterize the deviation of mean curvature compared with Euclidean metric, its positivity and rigidity is proved by \cite{ST}. Another important quasi-local mass is Hawking mass, which played a key role in proving Penrose Inequality in \cite{Br} and \cite{HI}. Because the Willmore functional of a surface can be arbitrarily large, we cannot expect the positivity for arbitrary surface. But for a stable $CMC$ sphere in nonnegative scalar curvature manifold, the Hawking mass is nonnegative \cite{CY}.

Bartnik in \cite[P.235]{Ba} proposed the rigidity problem of Hawking mass, i.e. what can we say about the ambient manifold when the Hawking mass vanishes for some surface. This paper devotes to partial result for rigidity problem if the surface is nearly round. We study the eigenvalue and eigenfunctions of Jacobi operator for stable $CMC$ surface with zero Hawking mass, then transfer the rigidity problem to a mean field type equation with respect to the second eigenvalue 6 of standard $S^2$ under some restriction. If the equation has only zero solution, then the rigidity of Hawking mass holds. We get the local uniqueness by studying the spherical harmonics on $S^2$ carefully and also iteration methods. If the solution is small in some sense, we can get the power decay of both kernel part of $\Delta + 6$ and also the orthogonal part. But we believe that the equation has only zero solution with the integral restriction.

The main term in Hawking mass is the Willmore functional, in $\mathbb{R}^3 $ the Willmore functional is constant 4$\pi$ if and only if the surface is round sphere. So we can detect the curvature of ambient space by Willmore functional. For this reason, we expect that manifold with zero Hawking mass surface may have some flatness properties.
\begin{theorem}\label{cmc rigidity}
Let (M,g) be a complete Riemnnian three manifold with scalar curvature $R(g)\geq0$, $\Omega \subset M$ be a simply connected domain with boundary $\Sigma=\partial\Omega$, if $\Sigma$ is a nearly round stable CMC sphere in M with $m_H(\Sigma)= 0$, then $\Omega$ isometric to a Euclidean ball in $\mathbb{R}^3$.
In particular, $\Sigma$ is isometric to the standard $\mathbb{S}^2$ in $\mathbb{R}^3$.
\\In this paper, nearly round is in the sense that Gauss curvature $K_\Sigma$ is $C^0$ closed to $\sqrt{\frac{4\pi}{|\Sigma|}}$, i.e. $|K_\Sigma-\sqrt{\frac{4\pi}{|\Sigma|}}|_{C^0}<\delta_0$ for some $\delta_0<<1.$
\end{theorem}
The hyperbolic case of above rigidity is following:
\begin{theorem}\label{hyperbolic cmc rigidity}
Let (M,g) be a complete Riemnnian three manifold with scalar curvature $R(g)\geq-6$, $\Omega \subset M$ be a simply connected domain with boundary $\Sigma=\partial\Omega$, if $\Sigma$ is a nearly round stable CMC sphere in M with $m_H(\Sigma)= 0$, then $\Omega$ isometric to a hyperbolic ball in $\mathbb{H}^3$.
\end{theorem}

By the examples of A. Carlotto and R. Schoen \cite{CS} there are manifolds with nonnegative scalar curvature and flat in half space of $\mathbb{R}^3$, so we can only expect flatness inside the surface with zero Hawking mass for stable $CMC$ surface. But we can get the global flatness for isoperimetric surface of sphere type:
\begin{theorem}\label{iso rigidity}
Let (M,g) be a complete AF three manifold with scalar curvature $R(g)\geq0$, if there exists a nearly round isoperimetric surface $\Sigma$ with $m_H(\Sigma)= 0$, then (M,g) is isometric to $(\mathbb{R}^3, \delta) $.
\end{theorem}
This theorem also have an hyperbolic version:
\begin{theorem}\label{hyperbolic iso rigidity}
Let (M,g) be a complete AH three manifold with scalar curvature $R(g)\geq-6$, if there exists a nearly round isoperimetric surface $\Sigma$ with $m_H(\Sigma)= 0$, then (M,g) is isometric to $(\mathbb{H}^3, g_{\mathbb{H}}) $.
\end{theorem}

We already know from \cite{CESY} that large surfaces of the canonical stable CMC foliation in \cite{HY,QT} are isoperimetric and close to the coordinate spheres. So we can get the rigidity result for large isoperimetric surface. For rigidity of small isoperimetric surface, we use the monotonicity of Hawking mass and also a rigidity result of Y. Shi.
\begin{theorem}\label{rigidity2}
Let (M,g) be a complete AF three manifold with scalar curvature $R(g)\geq0$, if there is an small enough isoperimetric surface $\Sigma$ with $m^+_H(\Sigma)= 0$ (see Definition \ref{maximal mass}), then (M,g) is isometric to $\mathbb{R}^3$.
\end{theorem}
\noindent{{\textbf{Structure of this paper}}}
In Section 2, we give the basic definitions. In Section 3, we prove the rigidity of Hawking mass for nearly round stable $CMC$ spheres. We transform the rigidity problem to a mean field type equation, and prove the local uniqueness of zero solution. By doing so, we get the surface with zero Hawking mass must be standard $\mathbb{S}^2$ and then use the rigidity of \cite{ST}\cite{ST07} to finish the proof. In Section 4, we prove the global properties of manifolds with nearly round isoperimetric surface have zero Hawking mass. This directly implies the rigidity for large isoperimetric surface in canonical stable $CMC$ foliation by Huisken-Yau \cite{HY} and Qing-Tian \cite{QT}. In Section 5, we prove the rigidity for small isoperimetric surface by using the monotonicity of Hawking mass. This relies on the fact that topology of small isoperimetric surface must be sphere. In Appendix \ref{sphere rigidity} we give the spherical harmonics and computations for the square of second order spherical harmonics. In Appendix \ref{existence} we give a sketch proof of existence of isoperimetric surfaces for all volumes in AF 3 manifolds. In Appendix \ref{continuity} we sketch the proof of continuity of isoperimetric profile for AF manifolds which is important to prove the right continuity of $I'_+$.

\section{Preliminaries}
We give some basic notations to present our result. Let $\Sigma \subset (M,g)$ be a surface with unit normal vector field $n$, second fundamental form $A$ and mean curvature $H$.
\begin{defn}
The Willmore functional of $\Sigma$ is defined by:
\begin{eqnarray}
W(\Sigma)=\frac{1}{4}\int_{\Sigma}{H^2}.
\end{eqnarray}
when $R(g)\geq0$.
\begin{eqnarray}
W(\Sigma)=\frac{1}{4}\int_{\Sigma}{(H^2-4)}.
\end{eqnarray}
when $R(g)\geq-6$.
\end{defn}
Willmore functional appears in various aspects, such as bending energy of elastic membranes. It is appears naturally in general relativity in form of Hawking mass of a surface:
\begin{defn}
The Hawking mass of $\Sigma$ is defined by:
\begin{eqnarray}
m_H(\Sigma)=\frac{{|\Sigma|^{\frac{1}{2}}}}{(16\pi)^{\frac{3}{2}}}(16\pi-\int_{\Sigma}{H^2}).
\end{eqnarray}
when $R(g)\geq0$.
\begin{eqnarray}
m_H(\Sigma)=\frac{{|\Sigma|^{\frac{1}{2}}}}{(16\pi)^{\frac{3}{2}}}(16\pi-\int_{\Sigma}(H^2-4)).
\end{eqnarray}
when $R(g)\geq-6$.
\end{defn}
\begin{defn}
If $H$ is constant along $\Sigma$, we say $\Sigma$ is a CMC surface;
\\The Jacobi operator of a CMC surface $\Sigma$ is the second variation of area:
\begin{eqnarray}
L_{\Sigma}=-\Delta_{\Sigma}-(|A|^2+Ric(n,n))
\end{eqnarray}

A CMC surface $\Sigma$ is $stable$ if the second eigenvalue of $L_{\Sigma}$ on mean zero functions is nonnegative
\begin{eqnarray}
\Lambda_2(L_{\Sigma})=inf\{\int_{\Sigma}fL_{\Sigma}f: \int_{\Sigma}f=0, \int_{\Sigma}f^2=1\}\geq 0
\end{eqnarray}
i.e. it satisfies the following stability condition:
\begin{eqnarray}
\int_{\Sigma}{(|A|^2+Ric(n,n))f^2}\leq \int_{\Sigma}{|\nabla f|^2}
\end{eqnarray}
for all $f\in C_c^\infty(\Sigma)$ and $\int_{\Sigma}f=0$.
\end{defn}
\begin{remark}
The above definition of eigenvalue in mean zero functions is different from the eigenvalue defined in the ordinary way by min-max construction:
\begin{eqnarray}\label{real eigenvalue}
\lambda_2(L_{\Sigma})=inf\{\int_{\Sigma}fL_{\Sigma}f: \int_{\Sigma}fu_0 = 0, \int_{\Sigma}f^2=1\}
\end{eqnarray}
where $u_0$ is the first eigenfunction of $L_{\Sigma}$. By definition we have
\begin{eqnarray}
\Lambda_2(L_{\Sigma})\leq \lambda_2(L_{\Sigma})
\end{eqnarray}

\end{remark}
We also want to study the isoperimetric surface in $AF(resp. AH)$ three manifold, we will always use the bracket to denote asymptotic hyperbolic case after related asymptotic flat situations.
\begin{defn}
A complete connected three manifold (M,g) is called AF(resp. AH), if there exists a constant $C> 0$, a compact set $K$, such that $M\setminus K$ is diffeomorphic to $\mathbb{R}^3\setminus B_R(0)$ for some $R>0$, and in standard coordinate the metric g has the following properties:
\begin{eqnarray}
g=\delta+h(resp. g=g_{\mathbb{H}}+h)
\end{eqnarray}
and
\begin{eqnarray}
|h_{ij}|+r|\partial h_{ij}|+r^2|\partial^2 h_{ij}|\leq Cr^{-\tau}
\end{eqnarray}
$\tau\in (\frac{1}{2},1](resp. \tau=3)$, where r and $\partial$ denote the Euclidean distance and standard derivative operator on $\mathbb{R}^3$ respectively. The region $M\setminus K$ is called the end of M.
\\The standard hyperbolic space $(\mathbb{H}^3, g_{\mathbb{H}}) $ is
\begin{eqnarray}
g_{\mathbb{H}}=\frac{1}{1+r^2}dr^2+r^2g_{\mathbb{S}^2}
\end{eqnarray}
\end{defn}

\begin{defn}
Given a complete Riemannian 3-manifold (M,g), its isoperimetric profile function $I : [0,\infty) \rightarrow [0,\infty) $ is defined by
\begin{eqnarray}
 \quad\quad\quad I(V ) = inf\{area(\partial \Omega) : \Omega \subset M \ is \ a \ compact \ region \ with \vol(\Omega) = V\}.
\end{eqnarray}
\end{defn}

\section{Rigidity of Hawking mass for nearly round stable CMC surfaces}
It was shown in \cite{CY} that the Hawking mass is nonnegative for a stable CMC sphere. It is proved by using Hersch type test function in stability condition and the nonnegativity of scalar curvature. Since we need to study the equality case, so we prove it here for completeness.
\begin{lemma}\label{CY} \cite{CY}
Let (M,g) be a Riemnnian three manifold with scalar curvature $R(g)\geq0$, if $\Sigma$ is a stable CMC sphere in M, then $m_H(\Sigma)\ge 0$.
\end{lemma}
\begin{proof}
By \cite{LY} there exist a conformal $\varphi: \Sigma\rightarrow S^2\subseteq \mathbb{R}^3$ with $\int_{\Sigma}{\varphi}=0$, we can plug these test functions in stability condition, use the fact that
\begin{eqnarray}
\int_{\Sigma}{|\nabla \varphi_i|^2}\geq\int_{\Sigma}{(|A|^2+Ric(n,n))\varphi_i^2}
\end{eqnarray}
for a surface conformal to $S^2\subseteq \mathbb{R}^3$ we have
\begin{eqnarray}
\int_{\Sigma}{|\nabla \varphi_i|^2}d\mu_{\Sigma}&=&\int_{S^2}{|\nabla x_i|^2} d\mu_{S^2}=-\int_{S^2}{x_i \triangle x_i} d\mu_{S^2} \nonumber\\
&=&2\int_{S^2}{x_i^2} d\mu_{S^2}=\frac{8}{3}\pi
\end{eqnarray}
Thus we can get
\begin{eqnarray}
8\pi\geq\int_{\Sigma}{|A|^2+Ric(n,n)}
\end{eqnarray}
By Gauss equation
\begin{eqnarray}
K_{\Sigma}=\frac{R}{2}-Ric(n,n)+\frac{1}{2}(H^2-|A|^2)
\end{eqnarray}
So we have
\begin{eqnarray}
|A|^2+Ric(n,n)&=&\frac{R}{2}-K_{\Sigma}+\frac{1}{2}(H^2+|A|^2)\nonumber\\
&=&\frac{1}{2}(R+|A^0|^2)+\frac{3}{4}H^2-K_{\Sigma}
\end{eqnarray}
where we have use that $|A|^2=|A^0|^2+\frac{1}{2}H^2$, we get

\begin{eqnarray}
8\pi\geq \frac{1}{2}\int_{\Sigma}{(R+|A^0|^2)+\frac{3}{4}\int_{\Sigma}H^2-\int_{\Sigma}K_{\Sigma}}
\end{eqnarray}
so we obtain
\begin{eqnarray}
16\pi-\int_{\Sigma}H^2\geq \frac{2}{3}\int_{\Sigma}{(R+|A^0|^2)}\geq 0
\end{eqnarray}
\end{proof}
We can get analogous results for hyperbolic case, see also \cite{Cho}:
\begin{lemma}\label{hyperbolic hawking mass}
Let (M,g) be a Riemnnian three manifold with scalar curvature $R(g)\geq-6$, if $\Sigma$ is a stable CMC sphere in M, then $m_H(\Sigma)\ge 0$.
\end{lemma}

Now we start to study the stable $CMC$ surface with zero Hawking mass. First we can get a spectral characterization of it. We need the following lemma in \cite{SI} which give a optimal estimate of the second eigenvalue of Schrodinger operator. It also gives part of the rigidity of second eigenvalue which is our case for Jacobi operator on a stable $CMC$ sphere.
\begin{lemma}\cite{SI}\label{cmc rigidity11}
For any continuous function $q$ on surface $\Sigma$, we have
\begin{eqnarray}
\lambda_2(-\Delta_\Sigma+q)|\Sigma|\leq 2A_c(\Sigma)+\int_{\Sigma}{q}
\end{eqnarray}
The equality holds iff $\Sigma$ admits a conformal map into the standard $S^2$ whose components are second eigenfunctions. If $\Sigma$ is of genus zero, then the equality implies that $\Sigma$ is conformal to the standard $S^2$ in $\mathbb{R}^3$ and $q$ is given by the energy density of a Moebius transform. Where $\lambda_2$ is the second eigenvalue of $-\Delta_\Sigma+q$ in the sence of (\ref{real eigenvalue}), $A_c(\Sigma)$ is the conformal volume in \cite{LY} and for sphere $A_c(\Sigma)=4\pi$.
\end{lemma}

By the above lemma we can have the following characterization of zero Hawking mass stable $CMC$ spheres.
\begin{pro}\label{eigenvalue}
Let (M,g) be a complete Riemnnian three manifold with scalar curvature $R(g)\geq0(resp. R(g)\geq-6)$, if $\Sigma$ is a stable CMC sphere with $m_H(\Sigma)= 0$ and area $|\Sigma|=4\pi$. Then the second eigenvalue $\lambda_2(-\Delta_{\Sigma}+K_{\Sigma})=3$, with three eigenfunctions $\varphi_1$,$\varphi_2$,$\varphi_3$, $\int_{\Sigma}{\varphi_i}=0$, and $\sum_{i=1}^{3}{\varphi_i}^2=1$. In particular, $|\nabla \varphi|^2=3-K_{\Sigma}$ which is independent of eigenfuncitons.
\end{pro}
\begin{proof}
From the above proof of Lemma \ref{CY} we can see if $m_H(\Sigma)= 0$ on $\Sigma$, we have $\int_{\Sigma}H^2=16\pi(resp. \int_{\Sigma}(H^2-4)=16\pi)$, $R=0(resp. R=-6),\ A^0=0$ on $\Sigma$. The area $|\Sigma|=4\pi$, then $H=2(resp. H=2\sqrt{2})$, the Jacobi operator become
\begin{eqnarray}
L_{\Sigma}=-\Delta_{\Sigma}+K_{\Sigma}-3
\end{eqnarray}
By the stability of $\Sigma$ and Lemma \ref{cmc rigidity11}, we have
\begin{eqnarray}
\quad 0\leq 4\pi \Lambda_2(L_{\Sigma})\leq 4\pi \lambda_2(L_{\Sigma})\leq 8\pi+\int_{\Sigma}(K_\Sigma-3)=0
\end{eqnarray}
so all the equality holds, in particular
\begin{eqnarray}
\lambda_2(-\Delta_\Sigma+K_\Sigma)=3
\end{eqnarray}
with three eigenfunctions $\varphi_1$,$\varphi_2$,$\varphi_3$, $\int_{\Sigma}{\varphi_i}=0$, and $\sum_{i=1}^{3}{\varphi_i}^2=1$.
\begin{eqnarray}\label{equality}
-\Delta_{\Sigma}\varphi+K_{\Sigma}\varphi-3\varphi=0
\end{eqnarray}
By $|\varphi|^2=\sum_{i=1}^{3}{\varphi_i}^2=1$, we have
\begin{eqnarray}
0=\Delta_{\Sigma}|\varphi|^2=2\varphi\Delta_{\Sigma}\varphi+2|\nabla \varphi|^2
\end{eqnarray}
Take inner product of $\varphi$ with (\ref{equality}), we get
\begin{eqnarray}\label{equality1}
|\nabla \varphi|^2=3-K_{\Sigma}
\end{eqnarray}
\end{proof}
\begin{remark}
We see from the above lemma that the second eigenvalue and eigenfunctions of Schrodinger operator $-\Delta_{\Sigma}+K_{\Sigma}$ equals with standard $S^2$, we expect that the metric is isometric to standard metric on $S^2$.
\end{remark}

In the following, we will always use $\Sigma$ to denote the stable $CMC$ surface with zero Hawking mass without extra explanation. Let $\varphi: \Sigma\rightarrow S^2\subseteq \mathbb{R}^3$ be the conformal map in Proposition \ref{eigenvalue} with $\int_{\Sigma}{\varphi}=0$. Denote the metric on $\Sigma$ is $g=e^u g_0$, $g_0$ being the standard metric on $S^2$. By definition of conformal map $\varphi$, we have
\begin{eqnarray}\label{equality 11}
 e^{-u}=\frac {1}{2}|\nabla \varphi|^2
\end{eqnarray}
The standard formula for Gauss curvature under a conformal change of the metric give
\begin{eqnarray}\label{equality 12}
K_\Sigma = e^{-u} (1 -\frac 12 \Delta_{g_0} u)
\end{eqnarray}
so (\ref{equality1}) gives
\begin{eqnarray}\label{main equation}
\Delta_{g_0} u = 6-6e^u
\end{eqnarray}
Also the volume preserving variation implies
\begin{eqnarray}\label{additional equation}
\int_{S^2}{x_ie^u}=0
\end{eqnarray}
So for this stable $CMC$ surface with zero Hawking mass $\Sigma$
\begin{eqnarray}
K_\Sigma-1 = e^{-u} (1 -3+3e^u)-1=2(1-e^{-u})
\end{eqnarray}
This means that if $u$ is $C^0$ close to 0, then $K_\Sigma$ is $C^0$ close to 1, which implies $\Sigma$ is nearly round. If we can prove (\ref{main equation})(\ref{additional equation}) admit only zero solution, then the stable $CMC$ surface with vanish Hawking mass isometric to standard $S^2$.

Equation of (\ref{main equation}) type has been studied in various aspects, such as prescribed Gaussian curvature \cite{KW}, mean filed model, Chern Simons Higgs model. This kind of equation may have bifurcation when approach the eigenvalues of $S^2$, so it may lost compactness. Ding, Jost, Li, Wang \cite{DJLW}\cite{DJLW1} have studied the equation at the first eigenvalue, Li \cite{Li} has initiated study of the existence of solutions by computing the Leray-Schauder topological degree, Lin compute the degree on $S^2$ in \cite{Lin} and surface of any genus \cite{CL}. But there is few work on the uniqueness of this kind of equation at second eigenvalue of $S^2$. In fact, the bifurcation will occur after the first eigenvalue, it is hard to guarantee the uniqueness in global. But we can get local uniqueness of the constant solution for (\ref{main equation}). That's why we put the nearly round condition in our results. We use the Lyapunov-Schmidt decomposition as in \cite{NT} to estimate the kernel of $\Delta_{g_0}+6$ and orthogonal part separately.

\begin{lemma}\label{uniqueness}
There exist $\delta_0>0$, such that if a solution of (\ref{main equation}) satisfies $sup|u|<\delta_0$, then $u\equiv0$.

\end{lemma}
\begin{proof}
In the following, the constant $C$ is universal, which may differ from line to line.
Denote $E_2=ker\{\Delta_{g_0}+6\}$, which is the second eigenspace of $-\Delta_{g_0}$ on standard $S^2$. It is well know that $E_2=span\{Y_{2,-2},Y_{2,-1},Y_{2,0},Y_{2,-1},Y_{2,2}\}$(see appendix below). Let $P_2$ be the projection operator on $E_2$.
Consider the decomposition $u=u_1+u_2$, where $u_1\in E_2^\perp$, and $u_2\in E_2$. Then
\begin{eqnarray}\label{main equation1}
\Delta_{g_0} u_1+6u_1 = 6(1+u-e^u)
\end{eqnarray}
\begin{eqnarray}\label{main equation2}
\Delta_{g_0} u_2+6u_2=0
\end{eqnarray}
As $(\Delta_{g_0}+6)^{-1}$ is bounded from $L^2$ to $W^{2,2}$ on $E_2^\perp$, we have
\begin{eqnarray}\label{eqn1}
|u_1|_{W^{2,2}}\leq C|1+u-e^u|_{L^2}
\end{eqnarray}
By the assumption, we can assume
\begin{eqnarray}\label{eqn2}
\sup|u|\leq \delta < 1
\end{eqnarray}

Then from (\ref{eqn1}) and Soblev embedding, we have
\begin{eqnarray}\label{eqn3}
|u_1|_{L^\infty}\leq C|u^2|_{L^2}\leq C\delta^2
\end{eqnarray}
Also from equation (\ref{main equation}), we know
\begin{eqnarray}
|\Delta_{g_0} u+6u+3u^2|_{L^2} = 6|1+u+\frac{1}{2}u^2-e^u|_{L^2}\leq C|u^3|_{L^2}\leq C\delta^3
\end{eqnarray}
By (\ref{eqn3}), we can get
$$ |u_1^2|_{L^2} \leq C \delta^4$$
By the ecomposition of $u_2=u-u_1$ we have
\begin{eqnarray}
|\Delta_{g_0} u+6u+3u_2^2|_{L^2} \leq 2|u^3|_{L^2}+6|u_1u|_{L^2}+3|u_1^2|_{L^2}\leq C\delta^3
\end{eqnarray}
In order to get the estimate of $u_2$, we project above equation to $E_2$, then we have
\begin{eqnarray}\label{eqn111}
|P_2u_2^2|_{L^2}\leq C\delta^3
\end{eqnarray}
By Lemma \ref{projection} below and (\ref{eqn111})we have
\begin{eqnarray}\label{eqn4}
|u_2|_{L^\infty}\leq C|u_2|_{L^2}\leq C\delta^{\frac{3}{2}}
\end{eqnarray}
Combine (\ref{eqn3})(\ref{eqn4}), we improve the initial assumption (\ref{eqn2})
\begin{eqnarray}\label{eqn5}
sup|u|\leq C|u|_{L^2}<C\delta^{\frac{3}{2}}
\end{eqnarray}
Take $\delta_0=\frac{1}{2}C^{-2}$ and iterate the procedure, we can get
\begin{eqnarray}\label{eqn6}
sup|u|\leq C_0|u|_{L^2}<C^{-2}(C^{2}\delta_0)^{(\frac{3}{2})^k}=C^{-2}(\frac{1}{2})^{(\frac{3}{2})^k}
\end{eqnarray}
and let $k\rightarrow\infty$, we get the desired result.
\end{proof}

\begin{lemma}\label{projection}
$\forall u_2\in E_2$, there have
\begin{eqnarray}
|P_2u_2^2|_{L^2}=\frac{1}{7}\sqrt{\frac{5}{\pi}}|u_2|^2_{L^2}
\end{eqnarray}
\end{lemma}
\begin{proof}
Let
\begin{eqnarray}
u_2=\sum_{i=-2}^{2}\lambda_iY_{2,i}
\end{eqnarray}
where $Y_{2,i}$ are the second order spherical harmonics(see Appendix \ref{shperical harmonics}). By computations and project $u_2^2$ to $E_2$, we have
\begin{eqnarray}
P_2u_2^2=\frac{1}{14}\sqrt{\frac{5}{\pi}}[2(\lambda_0^2-\lambda_{-2}^2-\lambda_{2}^2)+\lambda_1^2+\lambda_{-1}^2]Y_{2,0}
\\+\frac{1}{7}\sqrt{\frac{5}{\pi}}(\sqrt{3}\lambda_{-1}\lambda_1-2\lambda_{-2}\lambda_0)Y_{2,-2}\nonumber
\\+\frac{1}{14}\sqrt{\frac{5}{\pi}}[\sqrt{3}(\lambda_1^2-\lambda_{-1}^2)-4\lambda_{0}\lambda_2]Y_{2,2}\nonumber
\\+\frac{1}{7}\sqrt{\frac{5}{\pi}}[\lambda_{-1}\lambda_{0}+\sqrt{3}(\lambda_{-2}\lambda_{1}-\lambda_{-1}\lambda_2)]Y_{2,-1}\nonumber
\\+\frac{1}{7}\sqrt{\frac{5}{\pi}}[\lambda_{0}\lambda_1+\sqrt{3}(\lambda_{-2}\lambda_{-1}+\lambda_{1}\lambda_2)]Y_{2,1}\nonumber
\end{eqnarray}
Thus
\begin{eqnarray}\label{eqn11}
|P_2u_2^2|^2_{L^2}=(\frac{1}{7}\sqrt{\frac{5}{\pi}})^2\{\frac{1}{4}(2\lambda_0^2-2\lambda_{-2}^2-2\lambda_{2}^2+\lambda_{-1}^2+\lambda_1^2)^2\\
+(\sqrt{3}\lambda_{-1}\lambda_1-2\lambda_{-2}\lambda_0)^2
+\frac{1}{4}[\sqrt{3}(\lambda_1^2-\lambda_{-1}^2)-4\lambda_{0}\lambda_2]^2\nonumber\\
+[\lambda_{-1}\lambda_{0}+\sqrt{3}(\lambda_{-2}\lambda_{1}-\lambda_{-1}\lambda_2)]^2\nonumber
\\+[\lambda_{0}\lambda_1+\sqrt{3}(\lambda_{-2}\lambda_{-1}+\lambda_{1}\lambda_2)]^2\}\nonumber\\
=(\frac{1}{7}\sqrt{\frac{5}{\pi}})^2(\sum_{i=-2}^{2}\lambda_i^2)^2=(\frac{1}{7}\sqrt{\frac{5}{\pi}})^2|u_2|^2_{L^2}\nonumber
\end{eqnarray}

\end{proof}

The following rigidity result is some kind of  positive mass theorem in compact case(see \cite{Miao}, \cite{ST}, and \cite{HW}).
\begin{lemma}\label{sphere rigidity}
Let (M, g) be a compact orientable Riemannian 3-manifold with scalar curvature $R(g)\geq0$ and $\partial M$ isometric to round $S^2$ with mean curvature $H=2$. Then (M, g) is isometric to the unit ball in ($\mathbb{R}^3$, $\delta$).
\end{lemma}
To prove Theorem \ref{hyperbolic cmc rigidity} we need a rigidity result for hyperbolic case of sphere, see Theorem 3.8 in \cite{ST07} by Y. Shi and L. F. Tam.
\begin{lemma}\label{hyperbolic sphere rigidity}
Let (M, g) be a compact orientable Riemannian 3-manifold with scalar curvature $R(g)\geq-6$ and $\partial M$ isometric to round $S^2$ with mean curvature $H=2\sqrt{2}$. Then (M, g) is isometric to the unit ball in hyperbolic space $\mathbb{H}^3$.
\end{lemma}
After Lemma \ref{uniqueness}, Lemma \ref{sphere rigidity} and Lemma \ref{hyperbolic sphere rigidity}, now we are in the position to prove Theorem \ref{cmc rigidity} and Theorem \ref{hyperbolic cmc rigidity}.

\begin{proof}[Proof of Theorem \ref{cmc rigidity} and Theorem \ref{hyperbolic cmc rigidity}]
If $m_H(\Sigma)= 0$ on nearly round stable $CMC$ surface $\Sigma$, without loss of generality, assume $|\Sigma|=4\pi$, then $H=2(resp. H=2\sqrt{2})$
$$L_{\Sigma}=-\Delta_{\Sigma}+K-3$$
By Lemma \ref{uniqueness} we get the nearly round stable $CMC$ surface $\Sigma$ is standard $S^2$ in $\mathbb{R}^3$. Then by Lemma \ref{sphere rigidity}(resp. Lemma \ref{hyperbolic sphere rigidity}), we conclude that $\Omega$ isometric to unit ball in $\mathbb{R}^3(\mathbb{H}^3_{-1})$.
\end{proof}

Theorem \ref{cmc rigidity}(resp. Theorem \ref{hyperbolic cmc rigidity}) and Lemma \ref{CY}(resp. Lemma \ref{hyperbolic hawking mass}) can help us to understand Willmore functional in manifold with scalar curvature $R(g)\geq0(resp. R(g)\geq-6)$.
\begin{cor}
Let (M,g) be a complete Riemnnian three manifold with scalar curvature $R(g)\geq0(resp. R(g)\geq-6)$, $\Sigma=\partial \Omega$ is a stable CMC sphere, then $W(\Sigma)\leq 4\pi$.
\\If $\Sigma$ is nearly round, then equality holds if and only if $\Sigma$ is standard $S^2$ and $\Omega$ isometric to unit ball in $\mathbb{R}^3(resp. \mathbb{H}^3)$.
\end{cor}

\section{Rigidity of Hawking mass for nearly round isoperimetric surfaces}
Theorem 1 can be used to prove rigidity of isoperimetric surface in $AF$ manifold. By the manifold constructed by A. Carlotto and R. Schoen \cite{CS}, see also \cite{CESY}:
\begin{example}
There is an asymptotically flat Riemannian metric $g$ on $\mathbb{R}^3$ with non-negative scalar curvature and positive mass and such that $g=\delta$ on $\mathbb{R}^2\times (0, +\infty)$.
\end{example}
We can only expect flatness inside the surface with Hawking mass for stable $CMC$ surface. In order to prove Theorem \ref{iso rigidity} we need the following isoperimetric inequality  of \cite{Shi} which also plays a key role in proving the existence of isoperimetric surface for all volume in $AF$ three manifold. It says that if there exists a Euclidean ball in an AF manifold with nonnegative scalar curvature, then the AF manifold must be $\mathbb{R}^3$.
\begin{lemma}\cite{Shi} \label{shi rigidity}
Suppose (M, g) is an AF manifold with scalar curvature $R(g)\geq0$.
Then for any $V>0$ \begin{eqnarray}I(V)\leq(36\pi)^{\frac{1}{3}}V^{\frac{2}{3}}.\end{eqnarray}
There is a $V_0 > 0$ with
\begin{eqnarray}I(V_0)=(36\pi)^{\frac{1}{3}}V_0^{\frac{2}{3}}\end{eqnarray}
if and only if (M, g) is isometric to $\mathbb{R}^3$.
\end{lemma}
Also there has an analogous result for isoperimetric profile on AH manifold, see Propostion 3.3 in \cite{JSZ}.
\begin{lemma}\cite{JSZ} \label{hyperbolic shi rigidity}
Suppose (M, g) is an AH manifold with scalar curvature $R(g)\geq-6$.
Then for any $V>0$ \begin{eqnarray}I(V)\leq I_{\mathbb{H}}(V).\end{eqnarray}
There is a $V_0 > 0$ with
\begin{eqnarray}I(V_0)=I_{\mathbb{H}}(V_0)\end{eqnarray}
if and only if (M, g) is isometric to $(\mathbb{H}^3, g_{\mathbb{H}}) $.
\end{lemma}

Now we can prove the rigidity of nearly round isoperimetric surface:
\begin{proof}[Proof of Theorem \ref{iso rigidity} and Theorem \ref{hyperbolic iso rigidity}]
If there is an nearly round isoperimetric surface  $\Sigma$ with $m_H(\Sigma)= 0$, assume $|\Sigma|=4\pi$, then $H=2$. By Theorem \ref{cmc rigidity}, the isoperimetric region is a Euclidean ball of volume $\frac{4}{3}\pi$. So we have
\begin{eqnarray}
4\pi=I(\frac{4}{3}\pi)=(36\pi)^{\frac{1}{3}}(\frac{4}{3}\pi)^{\frac{2}{3}},
\end{eqnarray}
by the rigidity part of Lemma \ref{shi rigidity}, we conclude that $(M, g)$ is isometric to $\mathbb{R}^3$.
\\Theorem \ref{hyperbolic iso rigidity} is follows similarly by Theorem \ref{hyperbolic cmc rigidity} and Lemma \ref{hyperbolic shi rigidity}.
\end{proof}

In fact, large surfaces of the canonical stable CMC foliation in \cite{HY,QT} are isoperimetric and close to the coordinate spheres \cite{CESY}.

\begin{cor}
Let (M,g) be an AF three manifold with scalar curvature $R(g)\geq0$, then the Hawking mass of all the large enough surfaces in canonical stable $CMC$ foliation by in \cite{HY,QT} are positive unless (M,g) is isometric to $\mathbb{R}^3$.
\end{cor}

\section{Rigidity of Hawking mass for small isoperimetric surfaces}

For rigidity of small isoperimetric surface, we need to prove that it is a sphere when the volume is small enough.

\subsection{Topology of small isoperimetric surface}
It is known in \cite{Ros} that for a compact manifold without boundary, the isoperimetric surface is a topology sphere when enclosing volume is small enough contained in a geodesic ball. For AF manifolds we still have this property, the proof follows as compact case and relies on the behavior of infinity.
\begin{lemma}\label{small isopri}
If (M,g) is an AF three manifold, then there exits a $\delta_0>0$, such that for all volume $V\leq\delta_0$ the isoperimetric region is convex and contained in a small neighborhood of some point of M. In particular,
\begin{eqnarray}I(V)\sim (36\pi)^{\frac{1}{3}}V^{\frac{2}{3}} , when \ V \rightarrow 0.\end{eqnarray}

\end{lemma}
\begin{proof}
Let $\{\Sigma_n\}$ be a sequence of isoperimetric surfaces with second fundamental form $A_n$ and volume $V_n\rightarrow 0$. There have two possibilities:
\\1. $\{|A_n|\}$ is unbounded. Assume $r_n=max|A_n|=|A_n|(x_n)$, by scaling $\Sigma_n$ homothetically to $\Sigma'_n=r_n \Sigma_n$ with metric $g_n=r^2_ng$, also $r_n \rightarrow \infty$, $x_n \in \Sigma'_n$, second fundamental form of $\Sigma'_n$ satisfy $max|A'_n|=|A'_n(x_n)|=1$. We have $(M,x_n,g_n)\rightarrow (\mathbb{R}^3,0,\delta)$ smoothly, the limit manifold is standard $\mathbb{R}^3$ because the manifold is AF. Thus $\Sigma'_n$ is a sequence of stable $CMC$ surface with bounded curvature, locally $\Sigma'_n$ consists of certain number of sheets, each of them is graph over a bounded planar domain with bounded derivatives.
\\If two of the sheets become arbitrary close near some point when $n \rightarrow \infty$, then we can modify the surface to get a new one with smaller area and same volume. Hence by compactness results \cite{PR}, up to a subsequence, $\Sigma'_n \rightarrow \Sigma'$ smoothly with multiplicity one and $\Sigma' \subset \mathbb{R}^3$ is a stable CMC with constant $H_{\Sigma'}$ properly embedded in $\mathbb{R}^3$ endow with standard metric $\delta$, $0 \in \Sigma'$, $|A'(0)|^2=1$. By \cite{Sil}, $\Sigma'$ either a union of planes or a sphere. The curvature at origin is one imply that $\Sigma'$ is a unit sphere. Back to $\Sigma_n$, for $n$ large enough, the mean curvature $H_{\Sigma_n}$ of $\Sigma_n$ is large enough, such that
\begin{eqnarray}\frac{1}{2}H_{\Sigma_n}^2+Ric(n,n)>0.\end{eqnarray}
\\If $\Sigma_n$ is not connected, since the mean curvature of isoperimetric surface $\Sigma_n$ is same(see Appendix \ref{mean curvature}), for each component $\Sigma^i_n$, as $|A^i_n|^2\geq \frac{1}{2}H_{\Sigma^i_n}^2$, so we can get
\begin{eqnarray}|A^i_n|^2+Ric(n,n)>0\end{eqnarray}
 on the every component $\Sigma^i_n$. On the other hand, we can construct a variation$f_i$ on $\Sigma^i_n$ which is constant and $\sum_i{\int_{\Sigma^i_n}{f_i}}=0$ in the stability condition of isoperimetric inequality, this gives
 \begin{eqnarray}0\geq \sum_i f^2_i{\int_{\Sigma^i_n}{|A_n|^2+Ric(n,n)}},\end{eqnarray} a contradiction. So for large $n$, we know $\Sigma_n$ is connected and thus a sphere.
\\2. $\{|A_n|\}$ is bounded. Scaling $\Sigma_n$ to enclose volume 1. By the above argument we can get the limit consists of pairwise disjoint planes enclose volume 1, a contradiction. So the lemma follows.
\end{proof}
By above lemma, the rigidity follows from Theorem \ref{iso rigidity}. But it can also be proved by the monotonicity of Hawking mass with respect to volume of the connected isoperimetric surface. This method relies on the connectness of isoperimetric surface which used by \cite{Br}. Bray needed the connectness of isoperimetric surface when prove monotonicity of Hawking mass.
\subsection{Properties of $I$}

Isoperimetric profile $I$ contains important geometric information of the manifold. $I$ is nondecreasing in the outside of horizon. $I$ is concave if the manifold has nonnegative Ricci curvature. The existence and regularity properties of isoperimetic regions for $all$ volume for $AF$ is proved by \cite{Shi} combined with \cite{CCE}, we sketch the proof in Appendix \ref{existence} for completeness.

The continuity and differentiability of $I$ for $AF$ manifold is proved as in \cite{FN} for manifold with bounded geometry(Ricci curvature and volume of unit geodesic ball bounded below):
\begin{lemma}\label{conti and diff}
Given (M,g) is an AF manifold and $V \in (0, \infty)$, let $\Omega \subset M$ be
an isoperimetric region with $vol(\Omega) = V$ and denote $\partial \Omega$ as $\Sigma$. The isoperimetric profile has the following regularity:
\\a) I is continuous and have left and right derivatives at V , and $I'_+(V ) \leq H_\Sigma \leq I'_-(V )$, $I'_+(V )$ and $I'_-(V )$ are right and left continuous respectively.
\\b) $I''(V)I(V)^2 +\int_\Sigma(Ric(n,n)+|A_\Sigma|^2) \leq0$ holds in the sense of comparison functions, i.e. for every $V_0\geq0$, there exist a smooth function $I_{V_0}(V)\geq I(V)$, $I_{V_0}(V_0)=I(V_0)$, and $I_{V_0}''(V)I_{V_0}(V)^2 +\int_\Sigma(Ric(n,n)+|A_\Sigma|^2) \leq 0$.
\end{lemma}
\begin{proof}
The continuity of $I$ is proved in Appendix \ref{continuity} by adding and subtracting a small geodesic ball to the isoperimetric regions under the condition of bounded geometry.
We only prove (b) which implies the differentiability of $I$. For every $V_0>0$, assume $\Omega_0$ is the isoperimetric region with volume $V_0$ and $\Sigma_0=\partial \Omega_0$ is the isoperimetric surface with unit outer normal $n_0$ and second fundamental form $A_0$, mean curvature $H_0$. In order to get a upper bound of $I''$ we do a unit normal variation on $\Sigma_0$. Let $\Sigma_t$ denote the surface by flowing out $\Sigma_0$ with unit speed along the normal $n_0$ for time t. Since $\Sigma_0$ is smooth embedded surface, there exits a $\delta>0$ such that $\Sigma_t$ exists for any $t\in (-\delta, \delta)$. Let $I_{V_0}(t)=area(\Sigma_t)$, by the first and second variational formula of area we have:
\begin{align}
I'_{V_0}(t)&=\int_{\Sigma_t}{H}d{\mu}\\
V'(t)&=I_{V_0}(t)\\
H'(t)&=-|A|^2-Ric(n,n)
\end{align}
We can also parameterize these isoperimetric surface by its volume as $\Sigma(V)$, and $I_{V_0}(V)=area(\Sigma(V))$, by definition of $\Sigma(V_0)$, $I_{V_0}(V)\geq I(V)$, $I_{V_0}(V_0)=I(V_0)$, and we have
\begin{eqnarray}I'_{V_0}(V)=\frac{\int_{\Sigma(V)}{H}d{\mu}}{I_{V_0}(V)}=H
\end{eqnarray}
The second derivative of $I_{V_0}$ is
\begin{eqnarray}
I''_{V_0}(V)&=&\frac{\int_{\Sigma(V)}(H^2-|A|^2-Ric(n,n))d{\mu}}{I^2_{V_0}(V)}-\frac{H}{I^2_{V_0}(V)}\int_{\Sigma_t}{H}d{\mu} \nonumber \\&=&-\frac{\int_{\Sigma(V)}{|A|^2+Ric(n,n)}d{\mu}}{I^2_{V_0}(V)}.
\end{eqnarray}
For AF three manifold, Ricci curvature is bounded blow. Thus there exist $k\in R$, such that $Ric\geq kg$, it follows that
\begin{eqnarray}I''_{V_0}(V)\leq-\frac{k}{I_{V_0}(V)}\end{eqnarray}
If $k\geq 0$, then $I_{V_0}(V)$ is concave, by Lemma \ref{concave} below we can get the concaveness of $I(V)$, then the conclusion follows. In particular, $I'_+$, $I'_-$ are all nonincreasing functions,they are right and left continuous respectively, $I''$ exists almost everywhere.
\\If $k<0$, let $\lambda=\lambda(k,a,b)=:\frac{k}{2\delta(a,b)}$, where $\delta(a,b)=min\{I(V): V\in[a,b]\}$ strictly positive by continuity of $I$. For every $V_0\in[a,b]$, $I_{V_0}(V)+\lambda V^2\geq I(V)+\lambda V^2$, we get $I_{V_0}(V)+\lambda V^2$ is concave. We can argue as above to get the same conclusion.
\end{proof}
In the proof above, we have used the properties of concave function :
\begin{lemma}\label{concave}
 a) \cite{MJ}Let $f:(a,b)\rightarrow R$ be a continuous function. Then $f$ is concave if and only if for every $x_0 \in (a,b)$ there exists an open interval $I_{x_0}\subseteq (a,b)$ of $x_0$ and a concave smooth function $g_{x_0} : I_{x_0}\rightarrow R$ such that $g_{x_0}(x_0) = f(x_0)$ and $g_{x_0}(x) \geq f(x)$ for every $x_0\in I_{x_0}$.
 \\b) If $f:(a,b)\rightarrow R$ be a concave function, then $f'_+$ and $f'_-$ are monotonic nonincreasing functions and also right and left continuous respectively. Moreover, $f''$ exists almost everywhere.
\end{lemma}
\begin{proof}
a)If $f$ is concave, just take $g$ to be linear. If $f$ is not concave, then there exists $\epsilon>0$, such that $f_\varepsilon(x)=f(x)-\epsilon x^2$ is not concave. So we can choose $x_1,x_3\in(a,b)$, such that the graph of $f_\varepsilon(x)$ lies below line $l(x)$ from $(x_1, f_\epsilon(x_1))$ to $(x_3, f_\epsilon(x_3))$. Assume $f_\varepsilon(x)-l(x)$ attain its minimum at $x_2\in (x_1,x_3)$.
\\By hypothesis, there is a concave smooth $g_{x_2}(x)\geq f(x)$, and $g_{x_2}(x_2)= f(x_2)$. Then $g_\varepsilon(x)=g_{x_2}(x)-\epsilon x^2\geq f_\varepsilon(x)$, $g_\varepsilon(x_2)=f_\varepsilon(x_2)$, so we have that $g_\varepsilon(x)-l(x)$ also attain its minimum at $x_2\in (x_1,x_3)$ which implies $g''_\varepsilon(x_2)\geq 0$, but $g''_\varepsilon(x_2)=g''_{x_2}(x_2)-2\epsilon \leq -2\varepsilon$, a contradiction.
\\
\\b)It is well know that $f'_+$ and $f'_-$ are monotonic nonincreasing and $f''$ exists almost everywhere, so we just prove the right continuous of $f'_+$ and left continuous of $f'_-$ follows similarly. For any $x_0\in (a,b)$, by monotonicity of $f'_+$ have
\begin{eqnarray}\label{eq:1}\lim\limits_{x\to x^+_0}{f'_+(x)}\leq f'_+(x_0).\end{eqnarray}
On the other hand,
\begin{eqnarray}\label{eq:2}f'_+(x_0)=\lim\limits_{x\to x^+_0}\frac{f(x)-f(x_0)}{x-x_0}=\lim\limits_{x\to x^+_0}\frac{\int_{x_0}^{x}{f'_+(t)}dt}{x-x_0}
\end{eqnarray}
where we have used the stronger versions of the Fundamental Theorem of Calculus\cite{Wal}
\begin{eqnarray}\label{calculus}f(x)-f(x_0)=\int_{x_0}^{x}{f'_+(t)}dt\end{eqnarray}
whenever $f$ is continuous and $f'_+\in L^1$. Again by the monotonicity we have
\begin{eqnarray}\label{eq:3}f'_+(t)\leq \lim\limits_{x\to x^+_0}{f'_+(x)}
\end{eqnarray}
combined with (\ref{eq:2}) and (\ref{eq:3}) we get
\begin{eqnarray}\label{eq:4}f'_+(x_0)\leq\lim\limits_{x\to x^+_0}\frac{\int_{x_0}^{x}{\lim\limits_{x\to x^+_0}{f'_+(x)}}dt}{x-x_0}=\lim\limits_{x\to x^+_0}{f'_+(x)}\end{eqnarray}
Then (\ref{eq:1}) and (\ref{eq:4}) gives the right continuity of $f'_+$.

\end{proof}

\subsection{Monotonicity of $m^+_H$}
For differentiable point of $I$ we have $H(V)=I'(V)$, then we can replace $H$ with $I'$ in Hawking mass in order to simplify Hawking mass only as a function of volume. But $I$ may not be differentiable for every volume, and there is a jump for $H$ from $I'_+$ to $I'_-$ at the volume which is not differentiable. By the compactness of isoperimetric surface, see \cite{MMPR}, there is a surface which achieves the minimal(maximal) mean curvature enclosing the same volume. So we can define the maximal Hawking mass as:
\begin{defn}\label{maximal mass}
Let (M,g) be an AF three manifold with nonnegative scalar curvature, $\Sigma \subset M$ be a isoperimetric surface of volume V, then the maximal Hawking mass is
$m^+_H(V)=\sqrt{I(V)}(16\pi-I(V){I'_+(V)}^2).$

\end{defn}
$m^+_H$ is the maximal Hawking mass when $I$ is not differentiable and reduce to the ordinary Hawking mass at the differentiable point of $I$. We have the following monotonicity of $m^+_H$:
\begin{lemma}\cite{Br} \label{increasing}
Let (M,g) be an AF three manifold with nonnegative scalar curvature, assume for every $V>0$ there is a connected isoperimetric surface enclosing volume $V$, and also $I(V)$ is increasing, then $m^+_H(V)$ is nondecreasing outside the horizon.
\end{lemma}
\begin{proof}
By Gauss equation
\begin{eqnarray}
K=\frac{R}{2}-Ric(n,n)+\frac{1}{2}(H^2-|A|^2)
\end{eqnarray}
So we have
\begin{eqnarray}
|A|^2+Ric(n,n)=\frac{R}{2}-K+\frac{1}{2}(H^2+|A|^2)
\end{eqnarray}
by $|A|^2=|A^0|^2+\frac{1}{2}H^2$, and $R\geq0$,we have
\begin{eqnarray}
I''_{V_0}(V)=-\frac{\int_{\Sigma(V)}{|A|^2+Ric(n,n)}d{\mu}}{I^2_{V_0}(V)}\leq \frac{\int_{\Sigma(V)}{K-\frac{3}{4}H^2}d{\mu}}{I^2_{V_0}(V)}
\end{eqnarray}
\\By the connectness of $\Sigma(V)$, we have \begin{eqnarray}\int_{\Sigma(V)}Kd{\mu}=2\pi\chi(\Sigma(V))\leq4\pi,\end{eqnarray}
then \begin{eqnarray}\label{second order}I''_{V_0}(V)\leq \frac{16\pi-3I'_{V_0}(V)^2I_{V_0}(V)}{4I^2_{V_0}(V)}\end{eqnarray}
As we have proved that $I'_+(V)$ is right continuous, so is maximal Hawking mass. Thus it is sufficient to prove $m^+_H(V)$ is weak nondecreasing, i.e. for any $[a,b]\in(0,\infty)$, $\int_{a}^{b}{m^+_H(V)\phi'(V)}dV\leq0$ for all smooth nonnegative $\phi\in C_c^\infty(a,b)$,$\phi \geq 0$. The reason to do so is that $m^+_H(V)$ has only countable jump point. Let the difference quotient defined by $$\Delta^h F(V)=\frac{1}{h}(F(V+h)-F(V))$$
Then \begin{eqnarray}
\int_{a}^{b}{m^+_H(V)\phi'(V)}dV&=&\int_{a}^{b}{\sqrt{I(V)}(16\pi-I(V){I'_+(V)}^2)\phi'(V)}dV\nonumber\\
&=&\lim\limits_{h\to 0^+}{\int_{a}^{b}{\sqrt{I(V)}(16\pi-I(V){\Delta^h I(V)}^2)\Delta^h\phi(V)}dV}\nonumber\\
&=&-\lim\limits_{h\to 0^+}{\int_{a}^{b}\Delta^{-h}\{\sqrt{I(V)}(16\pi-I(V){\Delta^h I(V)}^2)\}\phi(V)}dV\nonumber\\
&=&\lim\limits_{h\to 0^+}{\int_{a}^{b}\{\phi I^{\frac{3}{2}}\{\Delta^{-h} (\Delta^h I)^2}-I'\frac{16\pi-3I'^2I}{2I^2}\}dV
\end{eqnarray}
\\where we use the fact that $I'_{+}=I'_{-}$ almost everywhere.
\\Since $I_{V_0}(V_0)=I(V_0)$, and $I_{V_0}(V)\geq I(V)$, also $I(V)$ is increasing, we can get $\Delta^{-h} (\Delta^h I)^2(V_0)\leq \Delta^{-h} (\Delta^h I_{V_0})^2(V_0)$, and $I'_{V_0}\geq0$, so

\begin{eqnarray}\int_{a}^{b}{m^+_H(V)\phi'(V)}dV&\leq&\lim\limits_{h\to 0^+}{\int_{a}^{b}\{\phi I_{V_0}^{\frac{3}{2}}\{\Delta^{-h} (\Delta^h I_{V_0})^2}-I_{V_0}'\frac{16\pi-3I_{V_0}'^2I_{V_0}}{2I_{V_0}^2}\}dV_{V_0}\nonumber\\
&=&\int_{a}^{b}{2\phi I_{V_0}^{\frac{3}{2}}I_{V_0}'\{I''_{V_0}}-\frac{16\pi-3I_{V_0}'^2I_{V_0}}{4I_{V_0}^2}\}dV_{V_0}\leq0
\end{eqnarray}
where we used (\ref{second order}) and the Fatou lemma for the last equality. Hence,$m^+_H(V)$ is nondecreasing.
\end{proof}

\begin{remark}
$\\$
1) Hawking mass is also monotonic along the stable CMC foliation as long as the area is nondecrasing, the proof is the same as above.
\\2) We can see that the monotonicity of maximal Hawking mass relies heavily on the connectness of isoperimetric surface. If the isoperimetric surface have more than one components, Bray in \cite{Br} consider the sum of three halves of area of the components $$F(V)=inf\{\sum_{i}Area(\Sigma_i)^{\frac{3}{2}}:\{\Sigma_i\} \ enclose \ volume \ V outside \ the \ horizons \}$$ under the condition the components are disjoint with each other, then he proved the mass
$$m^+(V ) = F(V )^\frac{1}{3}(36\pi-{F'_+}^2)/144\pi^\frac{3}{2}$$ is nondecreasing. In fact, for $F$ he got the estimate
\begin{eqnarray}F''(V)\leq \frac{36\pi-F'(V)^2}{6F(V)}\end{eqnarray}
then the prove follows as above. The minimizing surfaces are CMC generally with different mean curvatures on each component. When the minimizer of $F$ has only one component it must be isoperimetric surface. We have already know that for large enough volume in AF manifolds the isoperimetric surfaces are spheres close to coordinate spheres and  $m^+(V)=m^+_H(V)$, their limits are ADM mass of the manifold when volume goes to infinity.

\end{remark}
Now we are in the position to prove the rigidity of small isoperimetric surface:
\begin{proof}[Proof of Themrem \ref{rigidity2}]

 First we claim that \begin{eqnarray}\label{m0}\lim\limits_{V\to 0}m^+_H(V)=0.\end{eqnarray}
 In fact, by Lemma \ref{small isopri} we know the isoperimetric surface is of sphere type when volume small enough. Combined with Lemma \ref{CY} we get \begin{eqnarray}\label{m01}\lim\limits_{V\to 0}m^+_H(V)\geq0.\end{eqnarray}
 By definition, \begin{eqnarray}m^+_H(V)=\sqrt{I(V)}(16\pi-I(V){I'_+(V)}^2)\leq 16\pi \sqrt{I(V)}\end{eqnarray} which implies \begin{eqnarray}\label{m02}\lim\limits_{V\to 0}m^+_H(V)\leq0.\end{eqnarray}
 Thus the claim follows by (\ref{m01}) and  (\ref{m02}).
\\If there exists a isoperimetric surface $\Sigma$ with volume $0<V_0\leq\delta_0$, such that $m^+_H(V_0)=0$, then by monotonicity of Lemma \ref{increasing} for $m^+_H$ and (\ref{m0}), we get \begin{eqnarray}m^+_H(V)\equiv0, for \ any \ V\in[0,V_0], \end{eqnarray} thus
\begin{eqnarray}I(V){I'_+(V)}^2 \equiv\ 16\pi \ on \ [0,V_0].\end{eqnarray}
Since $I$ is continuous by Lemma \ref{conti and diff}, we get
\begin{eqnarray}I'_+(V)=I'(V)\ on \ [0,V_0].\end{eqnarray}
\\Then by strictly increasing of $I$ in \cite{Cho}, we get \begin{eqnarray}I'=\sqrt{\frac{16\pi}{I}},\end{eqnarray} since $I(0)=0$, we have \begin{eqnarray}I(V)=(36\pi)^{\frac{1}{3}}V^{\frac{2}{3}}\ on \ [0,V_0].\end{eqnarray} Then by Lemma \ref{shi rigidity} above we conclude that $(M,g)$ is isometric to $\mathbb{R}^3$.
\end{proof}

\section{Appendix}
\subsection{Spherical harmonics on $S^2$}\label{shperical harmonics}
$$\Delta_{S^2}=\frac{1}{\sin{\theta}}\frac{\partial}{\partial \theta}(\sin \theta \frac{\partial}{\partial \theta})+\frac{1}{\sin^2{\theta}}\frac{\partial^2}{\partial \varphi^2}$$
the eigenvalues of $-\Delta_{S^2}$ are
$$\lambda=l(l+1),l=0,1,2...$$
the eigenfunctions are
$$Y_l^m(\theta,\varphi)=\sqrt{\frac{2l+1}{4\pi}\frac{(l-|m|)!}{(l+|m|)!}}\sin^{|m|}{\theta}P_l^{|m|}(\cos \theta)e^{im\varphi}$$
where $m=-l,...,l$, and $P_l(x)$ is the Legendre polynomials,
$P_0(x)=1,P_1(x)=x,P_2(x)=\frac{1}{2}(3x^2-1)$
the reduction formula is
$$(n+1)P_{n+1}(x)=(2n+1)xP_n(x)-nP_{n-1}(x)$$
\textbf{The real form of spherical harmonics are}\\
$l=0$
$$Y_{0,0}=\frac{1}{2}\sqrt{\frac{1}{\pi}}$$
$l=1$
$$Y_{1,0}=\sqrt{\frac{3}{4\pi}}\cos\theta $$
$$Y_{1,-1}=\sqrt{\frac{3}{4\pi}}\sin\theta \sin \varphi$$

$$Y_{1,1}=\sqrt{\frac{3}{4\pi}}\sin\theta \cos \varphi$$
$l=2$
$$Y_{2,-2}=\frac{1}{4}\sqrt{\frac{15}{\pi}}\sin^2\theta \sin 2\varphi$$
$$Y_{2,-1}=\frac{1}{4}\sqrt{\frac{15}{\pi}}\sin2\theta \sin \varphi$$
$$Y_{2,0}=\frac{1}{4}\sqrt{\frac{5}{\pi}}(3\cos^2\theta-1) $$
$$Y_{2,1}=\frac{1}{4}\sqrt{\frac{15}{\pi}}\sin2\theta \cos \varphi$$
$$Y_{2,2}=\frac{1}{4}\sqrt{\frac{15}{\pi}}\sin^2\theta \cos 2\varphi$$
$l=4$
$$Y_{4,-4}=\frac{3}{16}\sqrt{\frac{35}{\pi}}\sin^4\theta \sin 4\varphi$$
$$Y_{4,-3}=\frac{3}{4}\sqrt{\frac{35}{2\pi}}\sin^3\theta\cos\theta \sin 3\varphi$$
$$Y_{4,-2}=\frac{3}{8}\sqrt{\frac{5}{\pi}}\sin^2\theta(7\cos^2\theta-1) \sin 2\varphi$$
$$Y_{4,-1}=\frac{3}{8}\sqrt{\frac{10}{\pi}}\sin\theta\cos\theta(7\cos^2\theta-3) \sin \varphi$$
$$Y_{4,0}=\frac{3}{16}\sqrt{\frac{1}{\pi}}(35\cos^4\theta-30\cos^2\theta+3)$$
$$Y_{4,1}=\frac{3}{8}\sqrt{\frac{10}{\pi}}\sin\theta\cos\theta(7\cos^2\theta-3) \cos \varphi$$
$$Y_{4,2}=\frac{3}{8}\sqrt{\frac{5}{\pi}}\sin^2\theta(7\cos^2\theta-1) \cos 2\varphi$$
$$Y_{4,3}=\frac{3}{4}\sqrt{\frac{35}{2\pi}}\sin^3\theta\cos\theta \cos 3\varphi$$
$$Y_{4,4}=\frac{3}{16}\sqrt{\frac{35}{\pi}}\sin^4\theta \cos 4\varphi$$
To compute $u_2^2$, we need to decompose the following terms into different order spherical harmonics:
$$Y_{2,0}^2=\frac{3}{7}\sqrt{\frac{1}{\pi}}Y_{4,0}+\frac{1}{7}\sqrt{\frac{5}{\pi}}Y_{2,0}+\frac{1}{4\pi}$$
$$Y_{2,-2}^2=-\frac{1}{2}\sqrt{\frac{5}{7\pi}}Y_{4,4}+\frac{1}{14}\sqrt{\frac{1}{\pi}}Y_{4,0}-\frac{1}{7}\sqrt{\frac{5}{\pi}}Y_{2,0}+\frac{1}{4\pi}$$
$$Y_{2,2}^2=\frac{1}{2}\sqrt{\frac{5}{7\pi}}Y_{4,4}+\frac{1}{14}\sqrt{\frac{1}{\pi}}Y_{4,0}-\frac{1}{7}\sqrt{\frac{5}{\pi}}Y_{2,0}+\frac{1}{4\pi}$$
$$Y_{2,-1}^2=-\frac{1}{7}\sqrt{\frac{5}{\pi}}Y_{4,2}-\frac{2}{7}\sqrt{\frac{1}{\pi}}Y_{4,0}-\frac{1}{14}\sqrt{\frac{15}{\pi}}Y_{2,2}+\frac{1}{14}\sqrt{\frac{5}{\pi}}Y_{2,0}+\frac{1}{4\pi}$$
$$Y_{2,1}^2=\frac{1}{7}\sqrt{\frac{5}{\pi}}Y_{4,2}-\frac{2}{7}\sqrt{\frac{1}{\pi}}Y_{4,0}+\frac{1}{14}\sqrt{\frac{15}{\pi}}Y_{2,2}+\frac{1}{14}\sqrt{\frac{5}{\pi}}Y_{2,0}+\frac{1}{4\pi}$$

$$Y_{2,-2}Y_{2,2}=\frac{1}{2}\sqrt{\frac{5}{7\pi}}Y_{4,-4}$$

$$Y_{2,-2}Y_{2,0}=\frac{1}{14}\sqrt{\frac{15}{\pi}}Y_{4,-2}-\frac{1}{7}\sqrt{\frac{5}{\pi}}Y_{2,-2}$$
$$Y_{2,2}Y_{2,0}=\frac{1}{14}\sqrt{\frac{15}{\pi}}Y_{4,2}-\frac{1}{7}\sqrt{\frac{5}{\pi}}Y_{2,2}$$
$$Y_{2,-1}Y_{2,0}=\frac{1}{7}\sqrt{\frac{15}{2\pi}}Y_{4,-1}+\frac{1}{14}\sqrt{\frac{5}{\pi}}Y_{2,-1}$$
$$Y_{2,1}Y_{2,0}=\frac{1}{7}\sqrt{\frac{15}{2\pi}}Y_{4,1}+\frac{1}{14}\sqrt{\frac{5}{\pi}}Y_{2,1}$$
$$Y_{2,-1}Y_{2,1}=\frac{1}{7}\sqrt{\frac{5}{\pi}}Y_{4,-2}+\frac{1}{14}\sqrt{\frac{15}{\pi}}Y_{2,-2}$$

$$Y_{2,-2}Y_{2,-1}=-\frac{1}{2}\sqrt{\frac{5}{14\pi}}Y_{4,3}-\frac{1}{14}\sqrt{\frac{5}{2\pi}}Y_{4,1}+\frac{1}{14}\sqrt{\frac{15}{\pi}}Y_{2,1}$$
$$Y_{2,2}Y_{2,1}=\frac{1}{2}\sqrt{\frac{5}{14\pi}}Y_{4,3}-\frac{1}{14}\sqrt{\frac{5}{2\pi}}Y_{4,1}+\frac{1}{14}\sqrt{\frac{15}{\pi}}Y_{2,1}$$

$$Y_{2,-2}Y_{2,1}=\frac{1}{2}\sqrt{\frac{5}{14\pi}}Y_{4,-3}-\frac{1}{14}\sqrt{\frac{5}{2\pi}}Y_{4,-1}+\frac{1}{14}\sqrt{\frac{15}{\pi}}Y_{2,-1}$$
$$Y_{2,2}Y_{2,-1}=\frac{1}{2}\sqrt{\frac{5}{14\pi}}Y_{4,-3}+\frac{1}{14}\sqrt{\frac{5}{2\pi}}Y_{4,-1}+\frac{1}{14}\sqrt{\frac{15}{\pi}}Y_{2,-1}$$
\subsection{Existence of isoperimetric surface for all volumes \cite{CCE}} \label{existence}
\begin{proof}
By Proposition 4.2 of \cite{EM}, for every $V>0$, there exists an isoperimetric region $\Omega$ and a radius $r\geq0$, such that
\begin{eqnarray}\label{em}
|\Omega|_g+\frac{4}{3}\pi r^3=V,\ |\partial \Omega|_g+4\pi r^2=I(V)
\end{eqnarray}
By the isoperimetric inequality of Shi \cite{Shi} on nonnegative scalar curvature manifold, we can get for every $r>0$, there is a bounded region $\Omega'$ with finite perimeter $|\partial \Omega'|_g$ lies arbitrary far out in the asymptotic flat region of $(M,g)$, such that
\begin{eqnarray}\label{em1}
|\partial \Omega'|_g=4\pi r^2,\ |\Omega'|_g>\frac{4}{3}\pi r^3
\end{eqnarray}
If $r>0$ in (\ref{em}), then there is $\Omega'$ satisfies (\ref{em1}). We consider the region $\Omega \cup \Omega'$, then
\begin{eqnarray}
|\Omega|_g+|\Omega'|_g>V,\ |\partial \Omega|_g+|\partial \Omega'|_g=I(V)
\end{eqnarray}
But by the defnition of $I$ and above equality we get
\begin{eqnarray}
I(|\Omega|_g+|\Omega'|_g)\leq|\partial \Omega|_g+|\partial \Omega'|_g=I(V)
\end{eqnarray}
By the strictly increasing of $I$ \cite{Cho}, we have
\begin{eqnarray}
I(|\Omega|_g+|\Omega'|_g)>I(V)
\end{eqnarray}
a contradiction. Thus $r=0$ which implies that $\Omega$ is the isoperimetric region of volume $V$.

\end{proof}
\subsection{ Continuity of $I$}\label{continuity}
\begin{proof}
The proof is from \cite{FN} for bounded geometry where they don't have existence of isoperimetric surface. We need to prove the upper semicontinuity and lower semicontinuity for $I$, i.e. for any $V_0>0$,
\begin{eqnarray}
\limsup\limits_{V\to V^+_0}{I(V)}\leq I(V_0),\ \limsup\limits_{V\to V^-_0}{I(V)}\leq I(V_0)
\end{eqnarray}
\begin{eqnarray}
I(V_0)\leq \liminf\limits_{V\to V^+_0}{I(V)},\ I(V_0)\leq \liminf\limits_{V\to V^-_0}{I(V)}
\end{eqnarray}
\textbf{Upper semicontinuity of $I$}: Given $ V_0>0$, there is isoperimetric region $\Omega_0$ such that $vol(\Omega_0)=V_0,\ area(\partial\Omega_0)=I(V_0)$.
For any $V\uparrow V_0$, we can subtract a small geodesic ball $B_r(p)$, such that $vol(B_r(p))=V_0 -V$,\ $vol(\Omega_0 \backslash B_r(p))=V$, thus we have
\begin{eqnarray}I(V)\leq area(\partial \Omega_0)+area(\partial B_r(p))=I(V_0)+area(\partial B_r(p))\end{eqnarray}
this implies \begin{eqnarray}\label{upper semi1}
\limsup\limits_{V\to V^+_0}I(V)\leq area(\partial \Omega_0)+\lim\limits_{V\to V^+_0}area(\partial B_r(p))=I(V_0)\end{eqnarray}

For any $V\downarrow V_0$, we can add a small geodesic ball $B_r(p)$, such that $vol(B_r(p))=V-V_0$,\ $vol(\Omega_0 \bigcup B_r(p))=V$, thus we have
\begin{eqnarray}
I(V)\leq area(\partial \Omega_0)+area(\partial B_r(p))=I(V_0)+area(\partial B_r(p))
\end{eqnarray}
this implies \begin{eqnarray}\label{upper semi2}
\limsup\limits_{V\to V^-_0}I(V)\leq area(\partial \Omega_0)+\lim\limits_{V\to V^-_0}area(\partial B_r(p))=I(V_0)
\end{eqnarray}
So we get the upper semicontinuity of $I$ from (\ref{upper semi1}) and (\ref{upper semi2}).\\
\textbf{Lower semicontinuity of $I$}: for  $V\uparrow V_0$, there exists isoperimetric region $\Omega$, such that $vol(\Omega)=V$, adding a small geodesic ball $B_r(p)$, such that $vol(B_r(p))=V_0 -V$, thus we have
\begin{eqnarray}
I(V_0)\leq area(\partial \Omega)+area(\partial B_r(p))=I(V)+area(\partial B_r(p))
\end{eqnarray}
this implies \begin{eqnarray}\label{lower semi1}
I(V_0)\leq \liminf\limits_{V\to V^-_0}I(V)+\lim\limits_{V\to V^-_0}area(\partial B_r(p))\leq \liminf\limits_{V\to V^-_0}I(V)
\end{eqnarray}
for  $V\downarrow V_0$, substract a small geodesic ball $B_r(p)$, such that $vol(B_r(p))=V -V_0$, thus we have
\begin{eqnarray}
I(V_0)\leq area(\partial \Omega)+area(\partial B_r(p))=I(V)+area(\partial B_r(p))
\end{eqnarray}
this implies \begin{eqnarray}\label{lower semi2}
I(V_0)\leq \liminf\limits_{V\to V^-_0}I(V)+\lim\limits_{V\to V^-_0}area(\partial B_r(p))\leq \liminf\limits_{V\to V^-_0}I(V)
\end{eqnarray}
The lower semicontinuity follows from (\ref{lower semi1}) and (\ref{lower semi2})
\end{proof}

\subsection{Mean curvature of isoperimetric surface} \label{mean curvature}
\begin{proof}
We know that isoperimetric surface is stable CMC and the mean curvature is same on each component. This follows by the stability condition when choosing piecewise constant variation function on each component. Assume $\Sigma=\Sigma_1\bigcup\Sigma_2$ is an isoperimetric surface with disjoint components $\Sigma_1$ and $\Sigma_2$. If the mean curvature of $\Sigma_1$ and $\Sigma_2$ is constant $H_1$ and $H_2$, respectively. Let
\begin{eqnarray}
f=\begin{cases}
-|\Sigma_2| & \text{on \ $\Sigma_1$}\\
|\Sigma_1| & \text{on \ $\Sigma_2$}
\end{cases}
\end{eqnarray}
As $\Sigma $ is an isoperimetric surface, so the first variation formula
\begin{eqnarray}
0&=&\int_{\Sigma}{fH}=\int_{\Sigma_1\bigcup\Sigma_2}{fH}\nonumber\\
&=&-|\Sigma_2|H_1 |\Sigma_1|+|\Sigma_1|H_2 |\Sigma_2|=|\Sigma_1| |\Sigma_2|( H_2-H_1)
\end{eqnarray}
So we have $H_1 =H_2$, which implies that mean curvature on each component is same.
\end{proof}

%\noindent{{\textbf{Proof of Lemma 1:}}}

\section*{Acknowledgement}
I would like to sincerely thank my advisor, Gang Tian, for his encouragement and constant support. I am very grateful to Yuguang Shi for suggesting this problem and for many stimulating discussions as well as pointing out the hyperbolic case for this paper. Many thanks to Prof. Ahmad El Soufi for explanation of Lemma \ref{cmc rigidity} in his paper. I also want to thank Jie Qing for pointing out the relation to eigenvalue problem, Jerry Kazdan and Yanyan Li for many discussions on mean field equation, Qian Wang and Zhongmin Qian for getting me interested in this area. Also I want to thank Wenshuai Jiang and Dongyi Wei for many useful comments and discussions. Additional thanks go to Jianchun Chu, Shiguang Ma, Dong Zhang, Haobin Yu and Wenlong Wang for interesting discussions on the subject.

\end{document}